\documentclass[11pt]{amsart}

\issueinfo{00}{0}{Month}{Year}
\copyrightinfo{2008}{University of Calgary}
\pagespan{1}{2}
\PII{ISSN 1715-0868}

\usepackage{graphicx}
\usepackage{epstopdf}
\usepackage{amsthm,amsmath,amssymb,amsxtra,amscd}
\usepackage{verbatim}
\usepackage{indent}
\usepackage{rotating}
\usepackage{multirow}
\usepackage{multicol}
\usepackage{url}
\usepackage{algorithmic}
\usepackage{xypic}

\newtheorem{theorem}{Theorem}[section]

\newtheorem{lemma}[theorem]{Lemma}

\newtheorem{definition}[theorem]{Definition}

\newtheorem{example}[theorem]{Example}
\newtheorem*{example*}{Example}

\newtheorem{observation}[theorem]{Observation}

\newtheoremstyle{myexample}{3pt}{3pt}{\rmfamily}{}{\itshape}{:}{ }{\thmname{#1}\thmnumber{ #2}\thmnote{ (#3)}}
\theoremstyle{myexample}

\newtheoremstyle{myremark}{3pt}{3pt}{\rmfamily}{}{\itshape}{:}{ }{\thmname{#1}}
\theoremstyle{myremark}

\newtheorem*{observation*}{Observation}

\newtheoremstyle{conjecture}{3pt}{3pt}{\itshape}{}{\bfseries}{.}{ }{\thmname{#1}\thmnote{ (#3)}}
\theoremstyle{conjecture}

\newtheorem*{question*}{Question}

\newtheorem{theorem*}{Theorem}

\numberwithin{equation}{section}

\newcounter{algorithm}
\renewcommand{\thealgorithm}{\thesection.\arabic{algorithm}}

\usepackage{graphicx}
\usepackage{amsmath}
\usepackage{amsthm}
\usepackage[hidelinks]{hyperref}

\def\Forb{\mathop{\mathrm{Forb}}\nolimits}

\newcommand{\overbar}[1]{\mkern 1.5mu\overline{\mkern-1.5mu#1\mkern-1.5mu}\mkern 1.5mu}

\def\str#1{\mathbf {#1}}

\def\K{{\mathcal K}}
\def\Fraisse{Fra\"{\i}ss\' e}
\def\role{role}

\begin{document}

\title{Completing graphs to metric spaces}

\author[A. Aranda]{Andr\'es Aranda}
\address{Computer Science Institute of Charles University (IUUK)\\ Charles University\\ Prague, Czech Republic}
\email{andres.aranda@gmail.com}

\author[D. Bradley-Williams]{David Bradley-Williams}
\address{Mathematisches Institut\\ Heinrich-Heine-Universit\"at\\ D\"usseldorf, Germany}
\email{david.bradley-williams@uni-duesseldorf.de}

\author[E. K. Hng]{Eng Keat Hng}
\address{Department of Mathematics\\ London School of Economics and Political Science\\ London, UK}
\email{e.hng@lse.ac.uk}

\author[J. Hubi\v cka]{Jan Hubi\v cka}
\address{Department of Applied Mathematics (KAM)\\ Charles University\\ Prague, Czech Republic}
\email{hubicka@iuuk.mff.cuni.cz}

\author[M. Karamanlis]{Miltiadis Karamanlis}
\address{Department of Mathematics\\ National Technical University of Athens\\ Athens, Greece}
\email{kararemilt@gmail.com}

\author[M. Kompatscher]{Michael Kompatscher}
\address{Department of Algebra\\ Charles University\\ Prague, Czech Republic}
\email{michael@logic.at}

\author[M. Kone\v cn\'y]{Mat\v ej Kone\v cn\'y}
\address{Department of Applied Mathematics (KAM)\\ Charles University\\ Prague, Czech Republic}
\email{matej@kam.mff.cuni.cz}

\author[M. Pawliuk]{Micheal Pawliuk}
\address{Department of Mathematics and Statistics\\ University of Calgary\\ Calgary, Canada}
\email{mpawliuk@ucalgary.ca}

\thanks{David Bradley-Williams is a member of the research training group GRK 2240 funded by the German Science Foundation (DFG). Jan Hubi\v cka and Mat\v ej Kone\v cn\'y are supported by project 18-13685Y of the Czech Science Foundation (GA\v CR). Jan Hubi\v cka is also supported by Charles University project Progres Q48. Michael Kompatscher is supported by the grant P27600 of the Austrian Science Fund (FWF) and Charles University Research Centre programs PRIMUS/SCI/12 and UNCE/SCI/022. Mat\v ej Kone\v cn\'y is also supported by Charles University, project GA UK No 2017--260452 (SVV)}
\begin{abstract}
We prove that certain classes of metrically homogeneous graphs omitting triangles of odd short perimeter as well as triangles of long perimeter have the extension property for partial automorphisms and we describe their Ramsey expansions.
\end{abstract}
\maketitle
\markleft{A. ARANDA ET AL.}
\centerline{Dedicated to Norbert Sauer on the occasion of his 70th birthday.}

\section{Introduction}
Given positive integers $\delta$, $K$ and $C$ we consider the class
$\mathcal A^\delta_{K,C}$ of finite metric spaces $\str{M}=(M,d)$ with integer
distances such that:
\begin{itemize}
 \item $d(x,y)\leq \delta$ for every $x,y\in M$ (the parameter $\delta$ is the {\em diameter} of $\mathcal A^\delta_{K,C}$);
 \item the perimeter of every triangle is less than $C$; and
 \item if a triangle has odd perimeter, then it is at least $2K+1$.
\end{itemize}
Here a {\em triangle} is any triple of distinct vertices $u,v,w\in M$ and its {\em perimeter} is $d(u,v)+d(v,w)+d(w,u)$.
We call parameters $\delta$, $K$ and $C$ {\em acceptable} if it
holds that $\delta\geq 2$, $1\leq K\leq \delta$ and $2\delta+K<C\leq 3\delta+1$
(this covers acceptable parameters in the sense of~\cite{Cherlin2013}, with the exception
of bipartite graphs).
Our main results can be stated as follows, with precise definitions to follow in sections \ref{sec:EPPA} and \ref{sec:RC}.
\begin{theorem}
\label{thm:EPPA}
The class $\mathcal A^\delta_{K,C}$ has the extension property for partial automorphisms (EPPA) for every acceptable choice of $\delta$, $K$ and $C$.
\end{theorem}
\begin{theorem}
\label{thm:ramsey}
The class $\overrightarrow{\mathcal A}^\delta_{K,C}$ of all linear orderings of metric spaces in $\mathcal A^\delta_{K,C}$ is a Ramsey class for every acceptable choice of $\delta$, $K$ and $C$.
\end{theorem}
We thus give new examples of classes of metric spaces which are Ramsey when enriched by linear orders and have EPPA, extending the lists of such classes obtained in
\cite{Nevsetvril2007,The2010,Hubicka2016}
 (for Ramsey classes) and
 \cite{Solecki2005,vershik2008,Conant2015}
 (for the EPPA). While these properties were historically treated independently, we show that both results follow from general constructions (stated as Theorems~\ref{thm:localfini} and \ref{thm:herwiglascar} below) and an analysis of an algorithm to fill the gaps in incomplete structures, thus turning them into metric spaces. Before going through the details, we will explain our motivation to consider these rather special-looking classes $\mathcal A^\delta_{K,C}$ and all the necessary notions.

\medskip

It has been observed by Ne\v set\v ril ~\cite{Nevsetvril1989a,Nevsetril2005,Kechris2005}
 that Ramsey classes have the so-called {\it amalgamation property}, which in turn implies the existence of a homogeneous \Fraisse{} limit (see section \ref{sec:homogen} for precise definitions). Therefore potential candidates for Ramseyness are usually taken from the well-known classification programme of homogeneous structures 
(see~\cite{Cherlin2013} for references). Cherlin recently extended
this list by a probably complete classification of classes with metrically homogeneous graphs as limits, where the 
$\mathcal A^\delta_{K,C}$ play an important \role{}.  For us it is particularly interesting, because the standard completion algorithm for metric spaces (which we introduce below) generally fails to produce metric spaces satisfying the $C$ bound. More complete results in this direction will appear in~\cite{Aranda2017}.

An {\em edge-labelled graph} $\str{G}$ is a pair $(G,d)$ where $G$ is the {\em vertex
set} and $d$ is a partial function from $G^2$ to $\mathbb N$ such that $d(u,v)=0$ if and only if $u=v$, and either $d(u,v)$ and $d(v,u)$ are both undefined or $d(u,v)=d(v,u)$ for each pair of vertices $u,v$ (equivalently, one could think of an edge-labelled graph as a relational structure). A pair of vertices $u,v$ such that $d(u,v)$ is defined is called an {\em edge} of $\str{G}$. We also call $d(u,v)$ the {\em length of the edge} $u,v$. We will refer to edge-labelled graphs simply as graphs when no confusion can arise; unless otherwise stated, subgraphs are assumed to be induced. 
A graph $\str{G}$ is {\em complete} if every pair of vertices
forms an edge and $\str{G}$ is called a {\em metric space} if the {\em triangle inequality} holds,
that is $d(u,w)\leq d(u,v)+d(v,w)$ for every $u,v,w\in G$.
A graph $\str{G}=(G,d)$ is {\em metric} if there exists a metric
space $\str{M}=(G,\bar{d})$ such that $d(u,v)=\bar{d}(u,v)$ for every edge $u,v$ of $\str{G}$.
Such a metric space $\str{M}$ is also called a {\em metric completion} of
$\str{G}$.

Given an edge-labelled graph $\str{G}=(G,d)$ the {\em path distance $d^+(u,v)$ of $u$ and $v$} is the minimum $\ell=\sum_{1\leq i\leq n-1}d(u_i,u_{i+1})$ over all possible sequences of vertices for which $u_1=u$, $u_n=v$ and $d(u_i,u_{i+1})$ is defined for every $i\leq n-1$. If there is no such sequence we put $\ell=\infty$.
It is well known that a connected graph $\str{G}=(G,d)$ is metric if and only if
$d(u,v)=d^+(u,v)$ for every edge of $\str{G}$. In this case $(G,d^+)$ is a metric completion of $\str{G}$
which we refer to as the {\em shortest path completion}.
This completion-algorithm also leads to an easy
characterisation of metric graphs. The graph $\str{G}$ is metric if and only if it does not contain a {\em non-metric cycle}, 
that is, an edge-labelled graph cycle such that one distance in the cycle is greater than the sum of the remaining distances. See e.g.~\cite{Hubicka2016} for details.
  In this paper we are going to introduce a generalisation of the shortest path completion.

\subsection{Extension property for partial automorphisms} \label{sec:EPPA}

Given two edge-labelled graphs $\str{G}=(G,d)$ and $\str{G}'=(G',d')$ a {\em homomorphism} $G\to G'$ is a function $f\colon G\to G'$
such that $d(x,y)=d'(f(x),f(y))$ whenever $d(x,y)$ is defined.
A homomorphism $f$ is an {\em embedding} (or {\em isometry} when the structures are metric spaces) if $f$ is one-to-one and $d(x,y)=d'(f(x),f(y))$ whenever either side
makes sense. A surjective embedding is an {\em isomorphism} and and {\em automorphism} is an isomorphism $\str{G}\to \str{G}$. A graph $\str{G}$ is an {\em (induced) subgraph} of $\str{H}$ if the identity mapping is an embedding $\str{G}\to \str{H}$.

A {\em partial automorphism} of an edge-labelled graph  $\str{G}$
is an isomorphism $f \colon \str{H} \to \str{H}'$ where $\str{H}, \str{H}'$ are subgraphs of $\str{G}$.  
We say that a class of finite structures $\K$  has the {\em extension property for
partial automorphisms}  ({\em EPPA}, sometimes called the {\em Hrushovski extension
property}) if whenever $\str{A} \in \K$ there is $\str{B} \in \K$ such that
$\str{A}$ is a subgraph of $\str{B}$ and such that
every partial automorphism of $\str{A}$ extends to an automorphism of $\str{B}$.

In addition to being a non-trivial and beautiful combinatorial property,
classes with EPPA have further interesting properties. For example, Kechris and
Rosendal~\cite{Kechris2007} have shown that the automorphism groups of their
\Fraisse{} limits are amenable.

In 1992, Hrushovski~\cite{hrushovski1992}
 showed that the class $\mathcal G$ of
all finite graphs has EPPA.  A combinatorial argument for Hrushovski's result
was given by Herwig and Lascar~\cite{herwig2000}
 along with a non-trivial
strengthening for certain, more restricted, classes of structures described by
forbidden homomorphisms. This result was independently used by
Solecki~\cite{Solecki2005} and Vershik~\cite{vershik2008} to prove EPPA for the
class of all finite metric spaces (with integer, rational or real distances -- for our presentation we will consider integer distances only).
 Recently Conant further developed this argument to generalised metric spaces~\cite{Conant2015}, where the distances are elements of a distance monoid. As a special case this implies EPPA for classes of metric spaces with distances limited to a given set
$S\subseteq \mathbb N$ and to metric graphs omitting short triangles of odd perimeter.

\subsection{Ramsey classes}\label{sec:RC}
For edge-labelled graphs $\str{A},\str{B}$ denote by ${\str{B}\choose \str{A}}$ the set
of all subgraphs of $\str{B}$ that are isomorphic to $\str{A}$.   A class
$\mathcal C$ of structures is a \emph{Ramsey class} if for every two objects $\str{A}$ and
$\str{B}$ in $\mathcal C$ and for every positive integer $k$ there exists a
structure $\str{C}$ in $\mathcal C$ such that the following holds: For every
partition of ${\str{C}\choose \str{A}}$ into $k$ classes there exists an
$\widetilde{\str B} \in {\str{C}\choose \str{B}}$ such that
${\widetilde{\str{B}}\choose \str{A}}$ is contained in a single class of the partition.

The notion of Ramsey classes was isolated in the 1970s and, being a
strong combinatorial property, it has found numerous applications, for example,
in topological dynamics~\cite{Kechris2005}.
It was independently proved by Ne\v set\v ril and R\"odl~\cite{Nevsetvril1976}
and Abramson--Harrington ~\cite{Abramson1978}
 that the class of all finite linearly ordered
hypergraphs is a Ramsey class. Several new classes followed.  We briefly outline the results
related to Ramsey classes of metric spaces.
In 2005 Ne\v set\v ril \cite{Nevsetvril2007} showed that the class of all finite
metric spaces   is a
Ramsey class when enriched by free linear ordering of vertices (see also \cite{masulovic2016pre} for alternative proof).

 This result was
extended to some subclasses $\mathcal A_S$ of finite metric spaces where all
distances belong to a given set $S$ by Nguyen Van Th{\'e}~\cite{The2010}.
Recently Hubi\v cka and Ne\v set\v ril further generalised this result to classes $\mathcal A _S$
for all feasible choices of $S$~\cite{Hubicka2016}
  that were earlier identified by Sauer~\cite{Sauer2013},
as well as to the class of metric spaces omitting triangles of short odd perimeter.

\subsection{Obstacles to completion}

The list of subclasses of metric spaces which are Ramsey closely corresponds to
the list of classes with EPPA. The similarity of the results is not a coincidence.  All of the
proofs proceed from a given metric space and, by a non-trivial construction, build
an edge-labelled graph with either the desired Ramsey property or EPPA.
The actual ``amalgamation engines'' have been isolated and
are based on characterising each class by a set of obstacles in the sense of the definition below.
Given a set $\mathcal O$ of edge-labelled graphs with edges bounded above by $\delta$, let $\Forb_{\delta}(\mathcal O)$ denote the class
of all finite or countably infinite edge-labelled graphs $\str{G}$ whose edges are bounded above by $\delta$ such that there is no $\str{O}\in \mathcal O$ with a homomorphism
$\str{O}\to \str{G}$.

Until now we have left the notion of completion only vaguely defined. Formally, a complete edge-labelled graph $\bar{\str{G}}$ is a completion of an edge-labelled graph $\str{G}$ if there is an injective homomorphism $\str{G}\to\bar {\str{G}}$. Given a class of edge-labelled graphs $\mathcal A$ and an edge-labelled graph $\str{G}$, we call $\bar {\str{G}}$ an $\mathcal A$-completion if $\bar {\str{G}}$ is a completion of $\str{G}$ and $\bar {\str{G}}\in\mathcal A$.

\begin{definition}
Given a class of metric spaces $\mathcal A$, we say that $\mathcal O$ is the {\em
set of obstacles of $\mathcal A$} if $\mathcal A\subseteq \Forb_{\delta}(\mathcal O)$ and
moreover every finite $\str{G}\in \Forb_{\delta}(\mathcal O)$ has a completion into $\mathcal A$.
\end{definition}

The following is a specialisation of main result of~\cite[Theorem~3.2]{herwig2000}:
\begin{theorem}[Herwig--Lascar~\cite{herwig2000}]
\label{thm:herwiglascar}
Given a finite set $\mathcal{O}$ of edge-la\-belled graphs, and $\str{G}\in \Forb_{\delta}(\mathcal O)$,
 if there exists some $\str{G}'\in \Forb_{\delta}(\mathcal O)$ such that $\str{G}$ is subgraph of $\str{G}'$ and every partial isomorphism between subgraphs of $\str{G}$
extends to an automorphism of $\str{G}'$ then there exists a finite such $\overline{\str{G}} \in \Forb_{\delta}(\mathcal O)$.
\end{theorem}
The following is a specialisation of Theorem 2.1 of ~\cite{Hubicka2016} (strong amalgamation is defined below):
\begin{theorem}[Hubi\v cka--Ne\v set\v ril~\cite{Hubicka2016}]
\label{thm:localfini}
Given a strong amalgamation class of finite metric spaces $\mathcal A$, assume that there exists a finite set of obstacles $\mathcal O$ of $\mathcal A$.
Then the class $\overrightarrow{\mathcal A}$ of all metric spaces from $\mathcal A$ along with a linear ordering of vertices is Ramsey.
\end{theorem}
 In the course of both proofs, incomplete edge-labelled graphs are produced.
Knowing the characterisation of obstacles and a completion algorithm, it is then
possible to turn such a graph into a metric space in the given class.

\subsection{The catalogue of metrically homogeneous graphs}\label{sec:homogen}
It is rather special for a class of structures to have a finite set of obstacles and a successful completion algorithm. Fortunately there
is an elaborate list of candidates which can be examined.
A weaker notion of completion, known as strong amalgamation, is well studied in the
context of \Fraisse{} theory. 

\begin{figure}
\centering
  \includegraphics{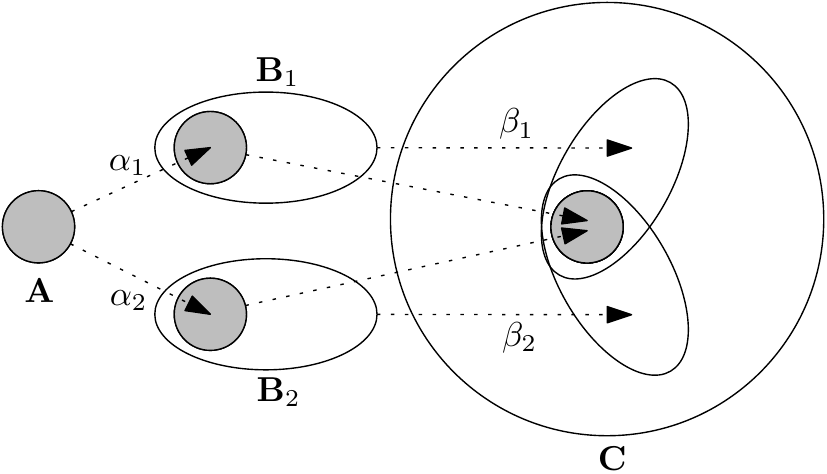}
\caption{An amalgamation of $\str{B}_1$ and $\str{B}_2$ over $\str{A}$.}
\label{amalgamfig}
\end{figure}%
Let $\str{A}$, $\str{B}_1$ and $\str{B}_2$ be edge-labelled graphs and $\alpha_1$ an embedding of $\str{A}$
into $\str{B}_1$, $\alpha_2$ an embedding of $\str{A}$ into $\str{B}_2$, then
every edge-labelled graph $\str{C}$
 with embeddings $\beta_1\colon\str{B}_1 \to \str{C}$ and
$\beta_2\colon\str{B}_2\to\str{C}$ such that $\beta_1\circ\alpha_1 =
\beta_2\circ\alpha_2$ is called an \emph{amalgamation} of $\str{B}_1$ and $\str{B}_2$ over $\str{A}$ with respect to $\alpha_1$ and $\alpha_2$. See Figure~\ref{amalgamfig}.

An \emph{amalgamation class} is a class $\K$ of finite edge-labelled graphs satisfying the following three conditions:
\begin{description}
\item[Hereditary property] For every $\str{A}\in \K$ and a subgraph $\str{B}$ of $\str{A}$ we have $\str{B}\in \K$;
\item[Joint embedding property] For every $\str{A}, \str{B}\in \K$ there exists $\str{C}\in \K$ such that $\str{C}$ contains both $\str{A}$ and $\str{B}$ as subgraphs;
\item[Amalgamation property]
For $\str{A},\str{B}_1,\str{B}_2\in \K$ and $\alpha_1$ an embedding of $\str{A}$ into $\str{B}_1$, $\alpha_2$ an embedding of $\str{A}$ into $\str{B}_2$, there is a $\str{C}\in \K$ which is an amalgamation of $\str{B}_1$ and $\str{B}_2$ over $\str{A}$ with respect to $\alpha_1$ and $\alpha_2$.
\end{description}

We say that an amalgamation is \emph{strong} when $\beta_1(x_1)=\beta_2(x_2)$ if and
only if $x_1\in \alpha_1(A)$ and $x_2\in \alpha_2(A)$.  Less formally, a strong
amalgamation glues together $\str{B}_1$ and $\str{B}_2$ with an overlap no
greater than the copy of $\str{A}$ itself. 

Classes with the amalgamation property give rise to homogeneous structures. Many
examples are provided by a well known classification programme
(see~\cite{Cherlin2013} for references). Every such class is a potential
candidate to be a Ramsey class, or a class having EPPA.  Cherlin recently extended
this list by a probably complete classification of classes with metrically homogeneous graphs as limits, where the 
$\mathcal A^\delta_{K,C}$ play an important \role{}.

From our perspective they are particularly interesting because they give a condition on the
largest perimeter of triangles. The shortest path completion typically violates this axiom and thus a new completion algorithm needs to be given.
While an amalgamation procedure is given in~\cite{Cherlin2013} it does not directly
generalise to a completion algorithm.

\section{Generalised completion algorithm}
\label{sec:algorithm}
In this section we will work with fixed acceptable parameters $\delta$, $K$ and $C$.  Put $\mathcal D=\{1,2,\dots \delta\}^2$. We will refer to elements of $\mathcal D$ as {\em forks}.

Consider a $\delta$-bounded variant of shortest path completion, where in the input graphs there are no distances greater than $\delta$ and in the output all edges longer than $\delta$ are replaced by an edge of that length. There is an alternative formulation of this completion: for a fork $\vec{f}=(a,b)$ define $d^+(\vec{f})=\min(a+b,\delta)$. In the $i$th step look at all incomplete forks $\vec{f}$ (i.e. triples of vertices $u,v,w$ such that exactly two edges are present) such that $d^+(\vec{f}) = i$ and define the length of the missing edge to be $i$.

This algorithm proceeds by first adding edges of length 2, then edges of length 3 and so on up to edges of length $\delta$ and has the property that out of all metric completions of a given graph, every edge of the completion yielded by this algorithm is as close to $\delta$ as possible.

It makes sense to ask what if, instead of trying to make each edge as close to $\delta$ as possible, one would try to make each edge as close to some parameter $M$ as possible. And for $M$ in certain range, it is indeed possible (made precise in Lemma \ref{lem:bestcompletion}). For each fork $\vec{f}=(a,b)$ one can define $d^+(\vec{f}) = a+b$ and $d^-(\vec{f}) = |a-b|$, i.e. the largest and the smallest possible distance that can metrically complete fork $\vec{f}$. The generalised algorithm will then complete $\vec{f}$ by $d^+(\vec{f})$ if $d^+(\vec{f})<M$, $d^-(\vec{f})$ if $d^-(\vec{f})>M$ and $M$ otherwise. It turns out that there is a good permutation $\pi$ of $\{1,\dots,\delta\}$, such that if one adds the distances in order given by the permutation, the generalised algorithm will indeed give a correct completion whenever possible. It is easy to check that the choice $M=\delta$ and $\pi=\{1,2,\dots,\delta\}$ gives exactly the shortest path completion algorithm.

In the following paragraphs we will properly state this idea, introduce some more rules in order to deal with the $C$ bound from the definition of $\mathcal A^{\delta}_{K,C}$ and prove that the algorithm works correctly.

\begin{definition}[Completion algorithm]\label{defn:ftmcompletion}
Given $c\geq 1$, $\mathcal F\subseteq \mathcal D$, and
an edge-labelled graph with distances at most $\delta$, we say that $\overbar{\str{G}}=(G,\bar{d})$ is the 
{\em $(\mathcal F,c)$-completion} of $\str{G}$ if the following hold.
\begin{enumerate}
\item If $u,v$ is an edge of $\str{G}$ it holds that $\bar{d}(u,v)=d(u,v)$.
\item If $u,v$ is not an edge of $\str{G}$ and there exist $(a,b)\in \mathcal F$ and $w\in G$ such that $\{d(u,w),d(v,w)\}=\{a,b\}$, we have that $\bar{d}(u,v)=c$.
\item There are no other edges in $\overbar{\str{G}}$.
\end{enumerate}

Given $1\leq M\leq \delta$, a one-to-one function $t\colon\{1,2,\ldots,\delta\}\setminus \{M\}\to \mathbb N$ 
and a function $\mathbb F$ from $\{1,2,\ldots,\delta\}\setminus \{M\}$ to the power set of $\mathcal D$, we define the {\em $(\mathbb F,t,M)$-completion} of  $\str{G}$
as the limit of a sequence of edge-labelled graphs  $\str{G}_1, \str{G}_2,\ldots$ such that $\str{G}_1=\str{G}$ and $\str{G}_{k+1}=\str{G}_k$ if $t^{-1}(k)$ is undefined
and $\str{G}_{k+1}$ is the $(\mathbb F(t^{-1}(k)),t^{-1}(k))$-completion of $\str{G}_{k}$ otherwise, with every pair of vertices not forming an edge in this limit set to distance $M$.
\end{definition}

We will call the vertex $w$ from Definition \ref{defn:ftmcompletion} the {\em witness of the edge $u,v$}. The function $t$ is called the {\em time function} of the completion because edges of length $a$ are inserted to $\str{G}_{t(a)}$ the $t(a)$-th step of the completion. If for a $(\mathbb F, t, M)$-completion and distances $a,c$ there is a distance $b$ such that $(a,b)\in \mathbb F(c)$ (i.e. the algorithm might complete a fork $(a,b)$ with distance $c$), we say that {\em $c$ depends on $a$}.

In the following paragraphs the letters $u,v,w$ will denote vertices and the letters $a,b,c$ will denote (lengths of) edges. We will slightly abuse notation and use the term triangle for both triples of vertices $u,v,w$ and for triples of edges $a,b,c$ with $a = d(u,v)$, $b = d(v,w)$, $c = d(u,w)$. The same convention will be used for forks.

\begin{definition}[Magic distances]
Let $M\in\{1,2,\dots, \delta\}$ be a distance. We say that $M$ is {\em magic} (with respect to $\mathcal A^\delta_{K,C}$) if $$\max\left(K, \left\lceil\frac{\delta}{2}\right\rceil\right) \leq M \leq \left\lfloor\frac{C-\delta-1}{2}\right\rfloor.$$
\end{definition}
When the parameters are acceptable, such an $M$ will exist.

\begin{observation}\label{obs:magicismagic}
The set of magic distances (for a given $\mathcal A^\delta_{K,C}$) is exactly the set $$S=\left\{1\leq a \leq \delta : \text{triangle }aab\text{ is allowed for all }1\leq b\leq \delta\right\}.$$
\end{observation}
\begin{proof}
If a distance $a$ is in $S$, then $a\geq K$ (because otherwise the triangle $aa1$ is forbidden by the $K$ bound), $a\geq \left\lceil\frac{\delta}{2}\right\rceil$ (because otherwise $aa\delta$ is non-metric) and $a\leq \left\lfloor\frac{C-\delta-1}{2}\right\rfloor$ (because otherwise $aa\delta$ is forbidden by the $C$ bound). The other implication follows from the definition of $\mathcal A^\delta_{K,C}$.
\end{proof}

An implementation in Sage of the following completion algorithm is available at~\cite{PawliukSage}.

Let $M$ be a magic distance and $1\leq x\leq \delta$ with $x\neq M$. Define 
\[
\begin{split}
\mathcal F^+_x &= \left\{(a,b)\in \mathcal D : a+b=x\right\}\\ 
\mathcal F^-_x &= \left\{(a,b)\in \mathcal D : |a-b|=x\right\}
\end{split}
\]
and $$\mathcal F^C_x = \{(a,b)\in \mathcal D : C-1-a-b=x\}.$$ We further define
$$\mathbb F_M(x) =
\begin{cases} 
      \mathcal F^+_x\cup \mathcal F^C_x & \text{if }x < M \\
      \mathcal F^-_x & \text{if }x > M.
\end{cases}
$$
For a magic distance $M$, we define the function $t_M\colon \{1,\dots,\delta\}\setminus \{M\} \rightarrow \mathbb N$ as
$$t_M(x) =
\begin{cases} 
      2x-1 & \text{if } x < M \\
      2(\delta-x) & \text{if }x > M.
\end{cases}
$$
Forks and how they are completed according to $\mathbb F_M$ are schematically depicted in Figure~\ref{fig:Fforks}.
\begin{figure}
\centering
\includegraphics{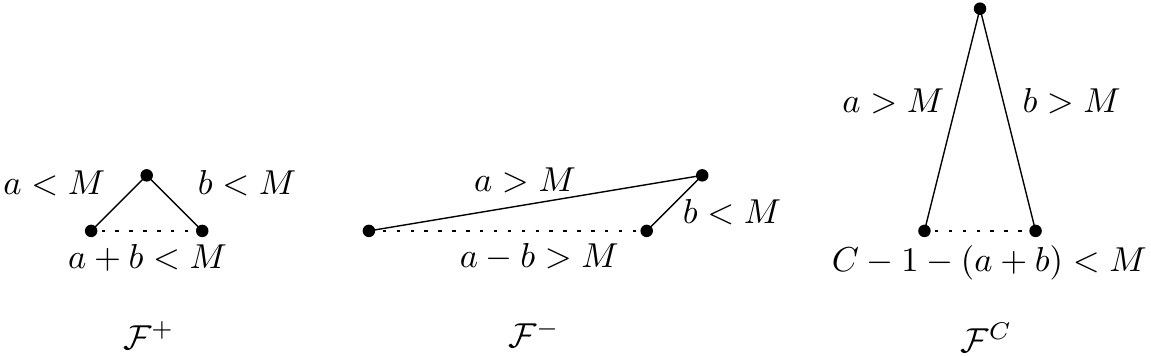}
\caption{Forks used by $\mathbb F_M$.}
\label{fig:Fforks}
\end{figure}

\begin{definition}[Completion with magic parameter $M$]
Let $M$ be a ma\-gic distance. We then call the $(\mathbb F_M,t_M,M)$-completion (of $\str{G}$) the {\em completion (of $\str{G}$) with magic parameter $M$}.
 We also use the name {\em completion algorithm with magic parameter $M$} to emphasise the process of iteratively adding distances.
\end{definition}
The
interplay of individual parameters of algorithm is schematically depicted in
Figure~\ref{fig:algorithm}.
\begin{figure}
\centering
\includegraphics{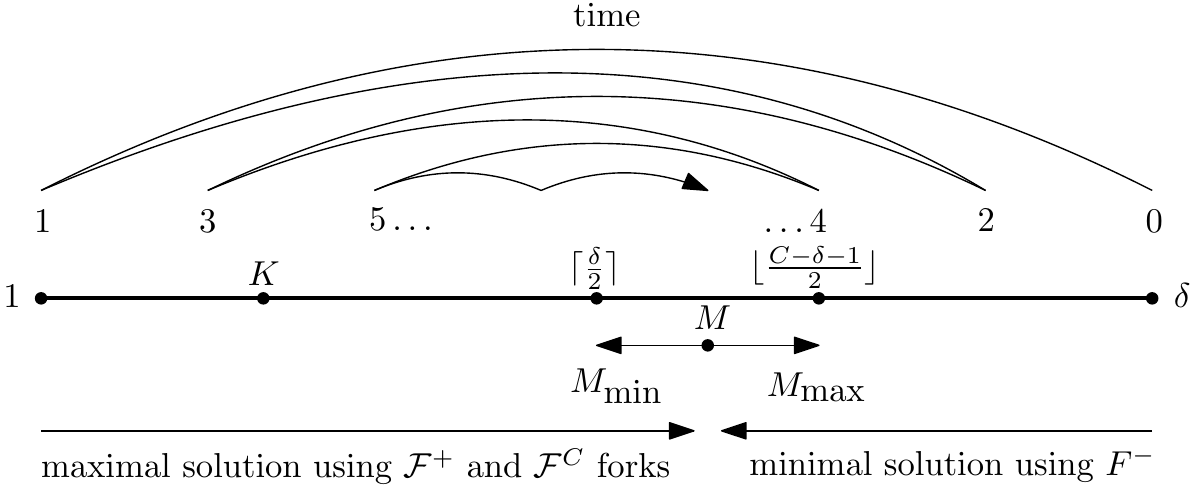}
\caption{A sketch of the main parameters of the completion algorithm.}
\label{fig:algorithm}
\end{figure}%
\begin{figure}
\centering
\includegraphics{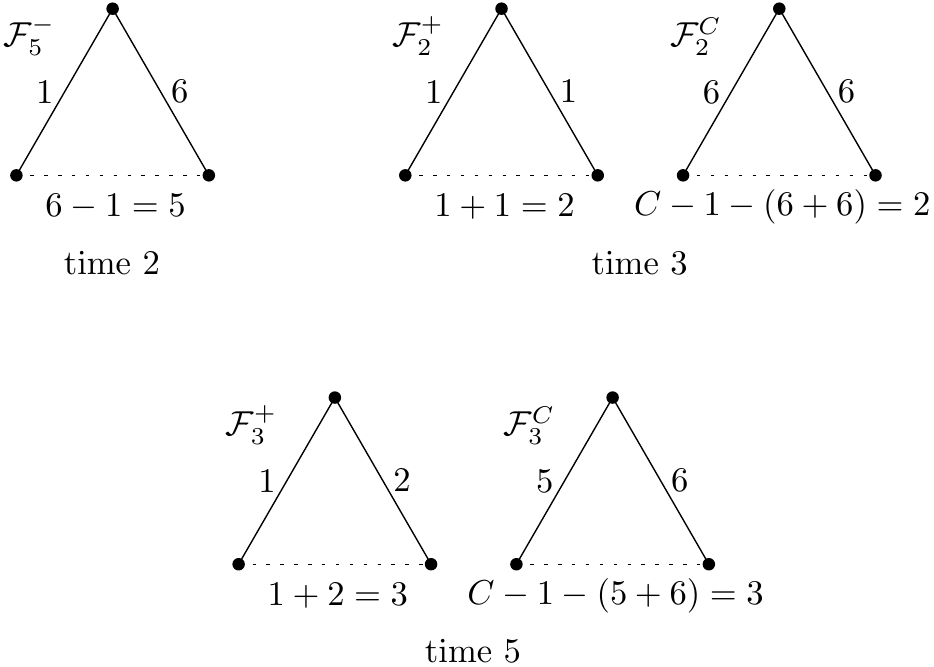}
\caption{Forks considered by the algorithm to complete to $\mathcal A^6_{2,15}$ with $M=4$.}
\label{fig:forks}
\end{figure}%
\begin{example}\rm
Consider $\delta=6, K=2, C=15$. Here $M$ can be chosen 3 or 4. We put $M=4$

Forbidden triangles are those that are non-metric (113, 114, 115, 116, 124, 125, 126, 135, 136, 146, 225, 226, 236), or rejected for the $K$-bound (111), or the $C$-bound (366, 466, 456, 555, 556, 566, 666).
\begin{table}[t]
\centering
\begin{tabular}{|c|c|c|c|c|c|c|}
\hline
&$j=1$&$j=2$&$j=3$&$j=4$&$j=5$ & $j=6$ \\ \hline
$i=1$&\textbf{2}&  $1,2,\textbf{3}$  &  $2,3,\textbf{4}$      & $3,\textbf{4},5$   		 & $\textbf{4},5,6$	   & $\textbf{5},6$  \\ 
$i=2$&					&  $1,2,3,\textbf{4}$&  $1,2,3,\textbf{4},5$  & $2,3,\textbf{4},5$ 		 & $3,\textbf{4},5,6$  & $\textbf{4},5,6$\\
$i=3$&					&  									 &  $1,2,3,\textbf{4},5,6$& $1,2,3,\textbf{4},5,6$ & $2,3,\textbf{4},5,6$& $3,\textbf{4},5$\\ 
$i=4$&					&  									 &  											&  $1,2,3,\textbf{4},5,6$& $1,2,3,\textbf{4},5$& $2,3,\textbf{4}$\\ 
$i=5$&					&  									 &  											&  									     & $1,2,3,\textbf{4}$  & $1,2,\textbf{3}$\\
$i=6$&					&  									 &  											&  											 &										 & $1, \textbf{2}$ \\ \hline
\end{tabular}
\caption{Possible ways to complete $(i,j)$ forks, the bold number is the completion with magic parameter $M = 4$.}
\label{tab:forks}
\end{table}
Table~\ref{tab:forks} lists all possible completions of forks, with the completion preferred by our algorithm in bold type. Completions for forks in this class are depicted in Figure~\ref{fig:forks}.
Those cases are the only forks where $M=4$ cannot be chosen, so instead the algorithm chooses the nearest possible completion. 

The algorithm will thus effectively run in four steps.
\begin{enumerate}
\item  At time 2 it will complete all forks $(1,6)$ with distance 5.
\item  At time 3 it will complete all forks $(1,1)$ and $(6,6)$ with distance 2.
\item  At time 5 it will complete all forks $(1,2)$ and $(5,6)$ with distance 3.
\item  Finally it will turn all non-edges into edges of distance 4.
\end{enumerate}
\begin{figure}
\centering
\includegraphics{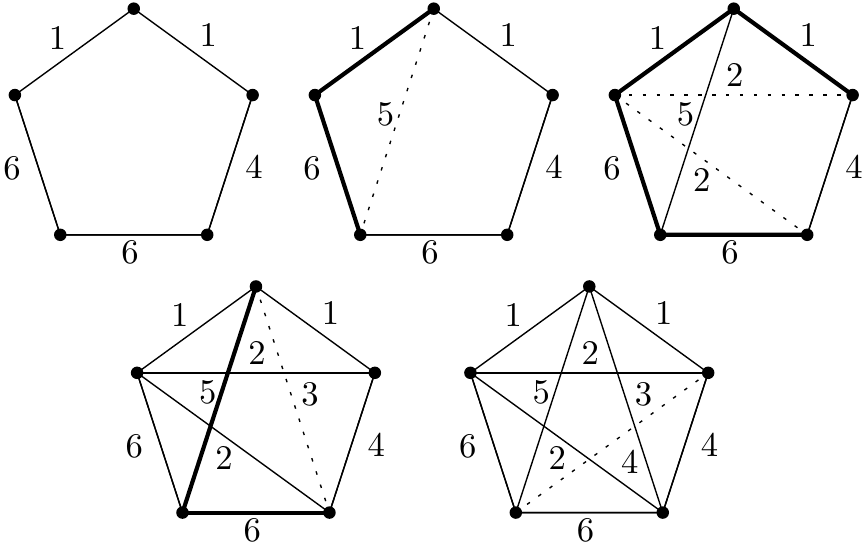}
\caption{Example of run of the algorithm.}
\label{fig:example}
\end{figure}%
 An example of runs of this algorithm is given in Figure~\ref{fig:example} and~\ref{fig:11665}.
\end{example}

The class $\mathcal A^\delta_{K,C}$ is given by forbidding those triangles with distances in $\{1, \ldots, \delta\}$ that satisfy one of the following conditions:
the non-metric condition (i.e. $abc$ is forbidden if $a+b<c$),
the $K$-bound condition ($a+b+c < 2K+1$ and $a+b+c$ is odd) and
the $C$-bound condition ($a+b+c \geq C$).
In our proof we will often consider these three classes of forbidden triangles separately. In the following we study how they are related to the magic parameter $M$. We will use $a,b,c$ for the lengths of the sides of the triangle and without loss of generality assume $a\leq b\leq c$. All conclusions are summarised in Figure~\ref{fig:Ftriangles}.
\begin{figure}
\centering
\includegraphics{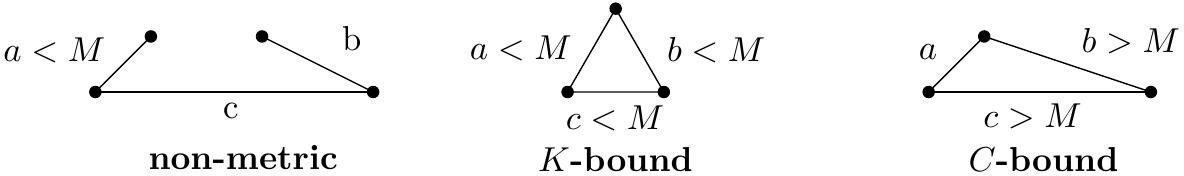}
\caption{Types of forbidden triangles.}
\label{fig:Ftriangles}
\end{figure}%
\begin{description}
\item[Non-metric] If $a+b<c$, then $a < M$, because otherwise $a+b\geq 2M \geq \delta$.

\item [$K$-bound] If $a+b+c < 2K+1$ and $a+b+c$ is odd and the triangle $abc$ is metric, then $a,b,c < K\leq M$, because if $c\geq K$, then from the metric condition $a+b\geq c\geq K$ and hence $a+b+c\geq 2K$, for odd $a+b+c$ this means $a+b+c\geq 2K+1$.

\item [$C$-bound]If $a+b+c\geq C$ then $b,c > M$. Suppose for a contradiction that $a,b\leq M$. We then have $a+b\geq C-c\geq C-\delta$, but on the other hand $a+b\leq 2M\leq 2\left\lfloor \frac{C-\delta-1}{2} \right\rfloor\leq C-\delta-1$, which together yield $C-\delta-1\geq C-\delta$, a contradiction.
\end{description}

Now we shall precisely state and prove that $t_M$ gives a suitable injection for the algorithm, as stated at the beginning of this section. The intuition behind the notion of ``$a$ depends on $b$" and the time function is that we wish to introduce edges to complete forks in a way that minimally reduces the options for subsequent forks.

\begin{lemma}[Time Consistency Lemma]\label{lem:expandtime}
Let $a,b$ be distances different from $M$. If $a$ depends on $b$, then $t_M(a) > t_M(b)$.
\end{lemma}
\begin{proof}
We consider three types of forks used by the algorithm:
\begin{description}
\item[$\mathcal F^+$ forks]
If $a<M$ and $\mathcal F^+_a\neq\emptyset$, then $b<a<M$, hence $t_M(b) < t_M(a)$.

\item[$\mathcal F^C$ forks]
If $a < M$ and $\mathcal F^C_a\neq\emptyset$, then we must have $b,c > M$. Otherwise, if for instance $b\leq M$, then $C-\delta-1\leq C-1-c = a+b < 2M \leq 2\left\lfloor \frac{C-\delta-1}{2} \right\rfloor$, a contradiction. As $C\geq 2\delta+2$, we obtain the inequality $b = (C-1)-c-a \geq (2\delta + 1) - \delta - a = \delta+1-a$. Hence $t_M(b) \leq 2(a-1) < 2a-1 = t_M(a)$.

\item[$\mathcal F^-$ forks]
Finally, we consider the case where $a > M$ and $\mathcal F^-_a\neq\emptyset$. Then either $a = b-c$, which implies $b>a>M$ and thus $t_M(b)<t_M(a)$, or $a = c-b$, which means $b = c-a\leq \delta-a$. Because of $a>M\geq \left\lceil\frac{\delta}{2}\right\rceil$, we have $b<M$. So $t_M(b) \leq 2(\delta-a) - 1 < 2(\delta-a) = t_M(a)$.
\end{description}
\end{proof}

The families $\mathbb F_M$ were chosen to include all forks that cannot be completed with $M$:

\begin{lemma}[$\mathbb F_M$ Completeness Lemma]\label{lem:misgood}
Let $\str{G}$ be an edge-labelled graph and $\overbar{\str{G}}$ be its completion with magic parameter $M$. If there is a forbidden triangle (w.r.t. $\mathcal A^\delta_{K,C}$) in $\overbar{\str{G}}$ with an edge of length $M$, then this edge was already in $\str{G}$.
\end{lemma}
\begin{proof}
By Observation \ref{obs:magicismagic} no triangle of type $aMM$ is forbidden, so suppose that there is a forbidden triangle $abM$ in $\overbar{\str{G}}$ such that the edge of length $M$ is not in $\str{G}$. For convenience define $t_M(M) = \infty$, which corresponds to the fact that edges of length $M$ are added in the last step.

\begin{description}
\item[Non-metric]
If $abM$ is non-metric then either $a+b<M$ or $|a-b|>M$. By Lemma \ref{lem:expandtime} we have in both cases that $t_M(a+b)$ (respectively, $t_M(|a-b|)$) is greater than $t_M(a)$ ($t_M(b)$). Therefore the completion algorithm would chose $a+b$, respectively $|a-b|$, as the length of the edge instead of $M$.

\item[$K$-bound]
Now that we know that $abM$ is metric, we also know that it is not forbidden by the $K$ bound, because $M\geq K$.

\item[$C$-bound]
If $abM$ is forbidden by the $C$ bound, then $t_M(C-1-a-b)>t_M(a),t_M(b)$ by Lemma \ref{lem:expandtime}, so the algorithm would set $C-1-a-b$ instead of $M$ as the length of the edge.
\end{description}
\end{proof}

The following lemma generalises the statement that the shortest path completion has all edges of maximum length possible. It will be the key ingredient for proving the correctness of the completion algorithm with magic parameter $M$.

\begin{lemma}[Optimality Lemma]\label{lem:bestcompletion}
Let $\str{G}=(G,d)$ be an edge-labelled graph such that there is a completion of $\str{G}$ into $\mathcal A^\delta_{K,C}$. Denote by $\overbar{\str{G}}=(G,\bar{d})$ the completion of $\str{G}$ with magic parameter $M$  and let $\str{G}'=(G,d')\in\mathcal A^\delta_{K,C}$ be an arbitrary completion of $\str{G}$. Then for every pair of vertices $u,v\in G$ either $d'(u,v) \geq \bar{d}(u,v) \geq M$ or $d'(u,v) \leq \bar{d}(u,v) \leq M$ holds.
\end{lemma}
\begin{proof}
Suppose that the statement is not true, and take any witness $\str{G}'=(G, d')$. Recall that the completion with magic parameter $M$ is defined as a limit of a sequence $\str{G}_1, \str{G}_2, \dots$ of edge-labelled graphs such that $\str{G}_1=\str{G}$ and each two subsequent graphs differ at most by adding edges of a single distance.

Take the smallest $i$ such that in the graph $\str{G}_i = (G,d_i)$ there are vertices $u,v\in G$ with $d_i(u,v) > M$ and $d_i(u,v) > d'(u,v)$ or $d_i(u,v) < M$ and $d_i(u,v) < d'(u,v)$. Let $w\in G$ be the witness of the edge $d_i(u,v)$.
We shall distinguish three cases, based on whether $d_{i}(u,v)$ was introduced by $\mathcal F^-$, $\mathcal F^+$ or $\mathcal F^C$:

\begin{description}
\item[$\mathcal F^-$ forks]
We have $M < d_i(u,v) = |d_{i-1}(u,w)-d_{i-1}(v,w)|$. Without loss of generality $d_{i-1}(u,w) > d_{i-1}(v,w)$, which means that $d_{i-1}(u,w) > M$ and $d_{i-1}(v,w) < M$ (as $M\geq \left\lceil\frac{\delta}{2}\right\rceil$). From the minimality of $i$ follows that $d'(u,w) \geq d_{i-1}(u,w)$ and $d'(v,w)\leq d_{i-1}(v,w)$. Since $\str{G'}$ is metric we have $d_i(u,v) = d_{i-1}(u,\allowbreak w)-d_{i-1}(v,w) \leq d'(u,w)-d'(v,w) \leq d'(u,v)$, which is a contradiction.

\item[$\mathcal F^+$ forks]
We have $M > d_i(u,v) = d_{i-1}(u,w)+d_{i-1}(v,w)$. Analogously to the first case we can show $d_{i-1}(u,w),\allowbreak d_{i-1}(v,\allowbreak w)<M$. By the minimality of $i$ we have $d'(u,w)\leq d_{i-1}(u,w)$ and $d'(v,w)\leq d_{i-1}(v,w)$. Since $\str{G'}$ is metric, we get $d'(u,v)\leq d_i(u,v)$, which contradicts to our assumptions.

\item[$\mathcal F^C$ forks]
We have $M > d_i(u,v) = C-1-d_{i-1}(u,w)-d_{i-1}(v,w)$. Recall that, by the acceptability of the parameters $\delta, K$ and $C$, we have $C-1\geq 2\delta+1$ and $M\leq \left\lfloor \frac{C-\delta-1}{2} \right\rfloor$. Thus we get $d_{i-1}(u,w),d_{i-1}(v,w)>M$ (otherwise, if, say, $d_{i-1}(u,w)\leq M$, we obtain the contradiction $C-\delta-1\geq 2M > d_{i-1}(u,w)+d_i(u,v) = C-1-d_{i-1}(v,w) \geq C-\delta-1$). So again $d'(u,w)\geq d_{i-1}(u,w)$ and $d'(v,w)\geq d_{i-1}(v,w)$, which means that the triangle $u,v,w$ in $\str{G}'$ is forbidden by the $C$ bound, which is absurd as $\str{G}'$ is a completion of $\str{G}$ in $\mathcal A^\delta_{K,C}$.
\end{description}
\end{proof}
\begin{lemma}[Automorphism Preservation Lemma]
\label{lem:aut}
For every metric graph $\str{G}$ and its completion with magic parameter $M$ (which we denote by $\str{M}$) it holds that
every automorphism of $\str{G}$ is also an automorphism of $\str{M}$.
\end{lemma}
\begin{proof}
Given $\str{G}$ and its automorphism $f\colon G\to G$, it can be verified by induction that for every $k>0$ $f$ is also an automorphism of graph $\str{G}_k$ as given in Definition~\ref{defn:ftmcompletion}.
That for every edge $x,y$ of $\str{G}_k$ which is not an edge of $\str{G}_{k-1}$ it holds that $f(x),f(y)$
is also an edge of $\str{G}_k$ which is not an edge of $\str{G}_{k-1}$ of the same distance. This follows
directly from the definition of $\str{G}_k$.
\end{proof}

In the next three lemmas we will use Lemma \ref{lem:bestcompletion} to show that an edge-labelled graph  $\str{G}$ has a completion into $\mathcal A^\delta_{K,C}$, if and only if the algorithm with magic parameter $M$ gives us such a completion. We will deal with each type of forbidden triangle separately, starting with the $C$ bound.

\begin{lemma}[$C$-bound Lemma]
Let $\str{G}=(G,d)$ be an edge-labelled graph such that there is a completion of $\str{G}$ into $\mathcal A^\delta_{K,C}$ and let $\overbar{\str{G}}=(G,\bar{d})$ be its completion with magic parameter $M$. Then there is no triangle forbidden by the $C$ bound in $\overbar{\str{G}}$.
\end{lemma}
\begin{proof}
Suppose for a contradiction that there is a triangle with vertices $u,v,w$ in $\overbar{\str{G}}$ such that $\bar{d}(u,v)+\bar{d}(v,w)+\bar{d}(u,w)\geq C$. For short let $a=\bar{d}(u,v)$, $b=\bar{d}(v,w)$ and $c=\bar{d}(u,w)$. Let $a',b',c'$ be the corresponding edge lengths in an arbitrary completion of $\str{G}$ into $\mathcal A^\delta_{K,C}$. Then two cases can appear. 

First assume that $a,b,c > M$. Then by Lemma \ref{lem:bestcompletion} we have $a'\geq a$, $b'\geq b$ and $c'\geq c$. Hence $a' + b' + c' \geq C$, which is a contradiction.

In the other case we can assume without loss of generality that $a\leq M$, $b\geq c>M$ and $a+b+c\geq C$. Again by Lemma \ref{lem:bestcompletion} we have that $b'\geq b$ and $c'\geq c$ and $a'\leq a$. If the edge $(u,v)$ was already in $\str{G}$, then clearly $a' + b' + c' \geq a+b+c\geq C$, which is a contradiction. If $(u,v)$ was not already an edge in $\str{G}$, then it was added by the completion algorithm with magic parameter $M$ in step $t(a)$. Let $\bar{a}=C-1-b-c$. Then clearly $\bar{a}<a$, which means that $t_M(\bar{a}) < t_M(a)$, and as $\bar{a}$ depends on $b,c$, we have $t_M(b),t_M(c)<t_M(\bar{a})$. But that means that the completion with magic parameter $M$ actually sets the length of the edge $u,v$ to be $\bar{a}$ in step $t_M(\bar{a})$, which is a contradiction.
\end{proof}

\begin{lemma}[Metric Lemma]
Let $\str{G}=(G,d)$ be an edge-labelled graph such that there is a completion of $\str{G}$ into $\mathcal A^\delta_{K,C}$ and let $\overbar{\str{G}}=(G,\bar{d})$ be its completion with magic parameter $M$. Then there is no non-metric triangle in $\overbar{\str{G}}$.
\end{lemma}
\begin{proof}
Suppose for a contradiction that there is a triangle with vertices $u,v,w$ in $\overbar{\str{G}}$ such that $\bar{d}(u,v)+\bar{d}(v,w)<\bar{d}(u,w)$. (Denote $a=\bar{d}(u,v)$, $b=\bar{d}(v,w)$ and $c=\bar{d}(u,w)$ and assume without loss of generality $a\leq b < c$.) Let $a',b',c'$ be the corresponding edge lengths in an arbitrary completion of $\str{G}$ into $\mathcal A^\delta_{K,C}$.

We shall distinguish three cases.
\begin{enumerate}
\item If $a,b,c < M$, then $t_M(a)\leq t_M(b) < t_M(a+b) < t_M(c)$, which means that $c$ must be already in $\str{G}$. But by Lemma \ref{lem:bestcompletion}, we have that $a' + b' \leq a  + b < c = c'$, which is a contradiction.

\item If $a,b<M$ and $c\geq M$, then by Lemma \ref{lem:bestcompletion} we have $a'\leq a$, $b'\leq b$ and $c'\geq c$, hence the triangle $a',b',c'$ is again non-metric.

\item If $a<M$ and $b,c\geq M$, then by Lemma \ref{lem:bestcompletion} we have $a'\leq a$ and $c'\geq c$. If $b$ was already in $\str{G}$, then $\str{G}$ has no completion -- a contradiction. Otherwise clearly $c-a > b \geq M$, so $t_M(c-a) < t_M(b)$. But as $c-a$ depends on $c$ and $a$, we get $t_M(c-a) > t_M(c),t_M(a)$, which means that the completion algorithm with magic parameter $M$ would complete the edge $v,w$ with the length $c-a$ and not with $b$.
\end{enumerate}
\end{proof}

\begin{lemma}[$K$-bound Lemma]
Let $\str{G}=(G,d)$ be an edge-labelled graph such that there is a completion of $\str{G}$ into $\mathcal A^\delta_{K,C}$ and let $\overbar{\str{G}}=(G,\bar{d})$ be its completion with magic parameter $M$. Then there is no triangle forbidden by the $K$ bound in $\overbar{\str{G}}$.
\end{lemma}
\begin{proof}
Suppose for a contradiction that there is a metric triangle with vertices $u,v,w$ in $\overbar{\str{G}}$ such that $\bar{d}(u,v)+\bar{d}(v,w)+\bar{d}(u,w)$ is odd and less than $2K+1$. Denote $a=\bar{d}(u,v)$, $b=\bar{d}(v,w)$ and $c=\bar{d}(u,w)$. As we observed above, $a,b,c < K\leq M$. Also assume without loss of generality, that $(u,v)$ was not an edge of $\str{G}$, but was added by the completion algorithm.

Notice that for any two distances $e,f$ it holds that $C-1-e-f \geq K$ (simply because $e,f\leq \delta$ and $C>2\delta+K$), so from the definition of $F_x^C$ we can see that no edge $a,b,c$ was added because of $F_x^C$, and as they are all small, they either were already in $\str{G}$ or they were added because of $F_x^+$.

Let $a',b',c'$ be the corresponding edges in some completion of $\str{G}$ into $\mathcal A^\delta_{K,C}$. Then Lemma \ref{lem:bestcompletion} implies that $a'\leq a$, $b'\leq b$ and $c'\leq c$ and thus $a'+b'+c'<2K+1$. So if there is a completion in which the triangle $u,v,w$ is not forbidden by the $K$ bound, it is because $a'+b'+c'$ is even. Since $a < M$ there is an $x$ such that $\bar{d}(u,x) + \bar{d}(v,x) = a$, if $a$ was not in $\str{G}$. But that means that $\bar{d}(u,x) + \bar{d}(v,x) + \bar{d}(u,v) = 2a < 2K$. Hence by changing the parity of $a$, the triangle $u,v,w$ becomes forbidden by the $K$ bound. The same argument can be made for the edges $b,c$ and that gives us a contradiction to $\str{G}$ having a completion into $\mathcal A^\delta_{K,C}$.
\end{proof}

From these lemmas we immediately get the following conclusion:

\begin{theorem} \label{thm:magiccompletion}
Let $\str{G}=(G,d)$ be an edge-labelled graph  such that there is a completion of $\str{G}$ into $\mathcal A^\delta_{K,C}$ and let $\bar{\str{G}}=(G,\bar{d})$ be its completion with magic parameter $M$. Then $\bar{\str{G}}\in\mathcal A^\delta_{K,C}$.

$\overbar{\str{G}}$ is optimal in the following sense:
Let $\str{G}'=(G,d')\in\mathcal A^\delta_{K,C}$ be an arbitrary completion of $\str{G}$ in  $\mathcal A^\delta_{K,C}$, then 
for every pair of vertices $u,v\in G$ either $d'(u,v) \geq \bar{d}(u,v) \geq M$ or $d'(u,v) \leq \bar{d}(u,v) \leq M$ holds.

Finally, every automorphism of $\str{G}$ is also an automorphism of $\overbar{\str{G}}$.
\end{theorem}

\section{Proofs of the main results}
Theorem \ref{thm:magiccompletion} implies the crucial lemma.
\begin{lemma}[Finite Obstacles Lemma]
\label{lem:obstacles}
For every acceptable choice of $\delta$, $K$ and $C$ the class $\mathcal A^\delta_{K,C}$ has a finite set of obstacles (which are all cycles of diameter at most $2^\delta 3$).
\end{lemma}

\begin{example}\rm
Consider $\mathcal A^6_{2,15}$ discussed in Section~\ref{sec:algorithm}.
The set of obstacles of this class contains all the forbidden triangles listed
earlier, but in addition to that it also contains some cycles with 4 or more vertices. A complete
list of those can be obtained by running the algorithm backwards from the forbidden
triangles.

All such cycles with 4 vertices can be constructed from the triangles by substituting distances by the forks depicted at Figure~\ref{fig:forks}. This means substituting $5$ for $16$ or $61$, $2$ for $11$ or $66$, $3$ for $12$, $21$, $56$ or $65$. With equivalent cycles removed this gives the following list:
$$
\begin{array}{rrcl}
\hbox{non-metric:}&113&\implies& 11\mathbf{12}^{**}, 11\mathbf{56}^{**}\\
&115 &\implies& 11\mathbf{16},  11\mathbf{61}^*\\
&124 &\implies& 1\mathbf{11}4,  1\mathbf{66}4\\
&125 &\implies& 1\mathbf{11}5,  1\mathbf{66}5,  12\mathbf{16},  12\mathbf{61}\\
&126 &\implies& 1\mathbf{11}6^*,  1\mathbf{66}6\\
&136 &\implies& 1\mathbf{12}6^*,  1\mathbf{21}6^*,  1\mathbf{56}6^*,  1\mathbf{65}6\\
&135 &\implies& 1\mathbf{12}5,  1\mathbf{21}5,  1\mathbf{56}5,  1\mathbf{65}5,  13\mathbf{16},  13\mathbf{61}\\
&225 &\implies& \mathbf{11}25^*,  \mathbf{66}25,  2\mathbf{11}5^*,  2\mathbf{66}5^*,  22\mathbf{16},  22\mathbf{61}^*\\
&226 &\implies& \mathbf{11}26^*,  \mathbf{66}26,  2\mathbf{11}6^*,  2\mathbf{66}6^*\\
&236 &\implies& \mathbf{11}36^*,  \mathbf{66}36,  2\mathbf{12}6,  2\mathbf{21}6^*,  2\mathbf{56}6^*,  2\mathbf{65}6\\
\hbox{$C$ bound:}&555 &\implies& \mathbf{16}55^*,  \mathbf{61}55^*,  5\mathbf{16}5^*,  5\mathbf{61}5^*,  55\mathbf{16}^*,  55\mathbf{61}^*\\
&366 &\implies& \mathbf{12}66^{**}, \mathbf{56}66^{**}\\
&456 &\implies& 4\mathbf{16}6^*,  4\mathbf{61}6\\
&556 &\implies& \mathbf{16}56^*,  \mathbf{61}56^*,  5\mathbf{16}6^*,  5\mathbf{61}6^*\\
&566 &\implies& \mathbf{16}66^*,  \mathbf{61}66^*
\end{array}
$$
Not all expansions here are necessarily forbidden, because not all of them correspond to a valid run of the algorithm.  However with the exception of cases denoted by $**$ all the above 4-cycles are forbidden. Also the list contains numerous duplicates which are denoted by $*$.

By repeating this procedure one obtains the following cycles with five and six vertices which cannot be completed into this class of metric graphs.
$$11116, 16616, 16661, 11115, 11665, 11216, 11261,$$
$$16615, 16165, 16561, 66665, 66216, 66261, 66666,$$
$$111116, 116616, 116661, 161616, 666616$$
Because there are no distances 5, 2 or 3 in the list of cycles with 6 edges, this completes the list of obstacles.

\begin{figure}
\centering
\includegraphics{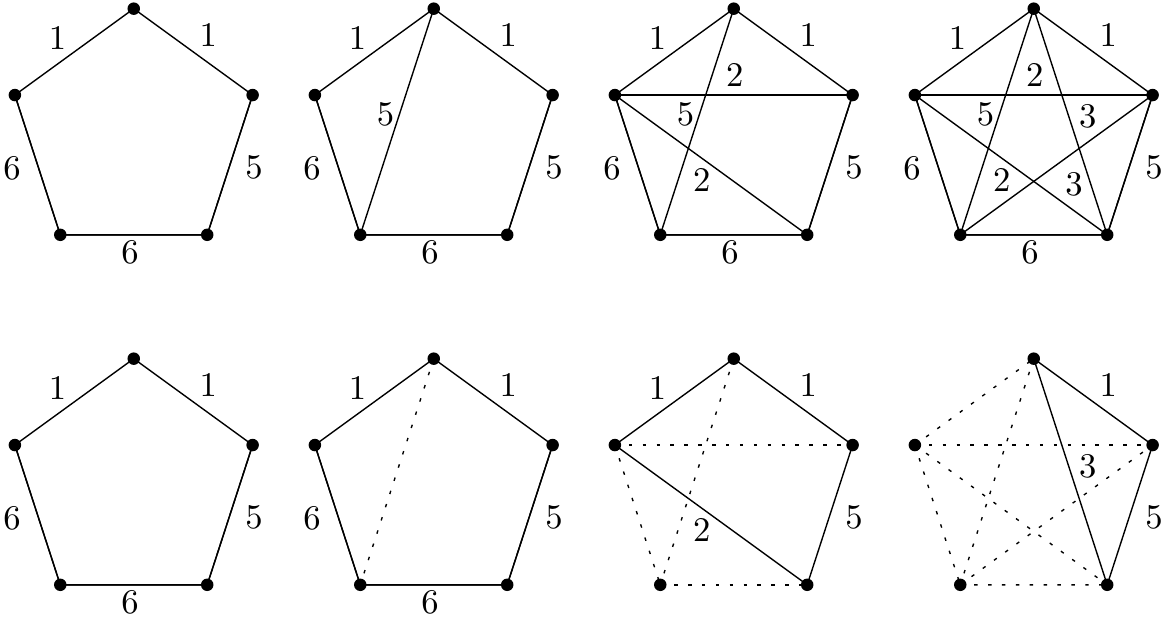}
\caption{Failed run attempting to complete the cycle 11665. In the bottom row is the backward run from non-metric triangle 135 to the original obstacle used in the proof of Lemma~\ref{lem:obstacles}.}
\label{fig:11665}
\end{figure}%
An example of a failed run of the algorithm trying to complete one of the forbidden cycles is depicted in Figure~\ref{fig:11665}.
\end{example}

\begin{proof}[Proof of Lemma \ref{lem:obstacles}]
Let $\str{G}=(G,d)$ be an edge-labelled graph with all distances at most $\delta$ and no completion in
$\mathcal A^\delta_{K,C}$.  We seek a subgraph of $\str{G}$ of bounded size
which has also no completion into $\mathcal A^\delta_{K,C}$.

Consider the sequence of graphs $\str{G}_0, \str{G}_1,\dots,\str{G}_{2M+1}$ as given by Definition~\ref{defn:ftmcompletion}
when completing $\str{G}$ with magic parameter $M$.  Set $\str{G}_{2M+2}$ to be the actual completion.

Because $\str{G}_{2M+2}\notin \mathcal A^\delta_{K,C}$ we know it contains a forbidden triangle $\str{O}$.
This triangle always exists, because $\mathcal A^\delta_{K,C}$ is 3-constrained.
By backward induction on $k=2M+1,2M,\dots, 0$ we obtain cycles $\str{O}_k$
of $\str{G}_k$ such that $\str{O}_k$ has no completion in $\mathcal A^\delta_{K,C}$
and there exists a homomorphism $f\colon \str{O}_k\to \str{G}_k$.

Put $\str{O}_{2M+1}=\str{O}$. By Lemma~\ref{lem:misgood} we know that this triangle is also in $\str{G}_{2M+1}$.
At step $k$ consider every edge $u,v$ of $\str{O}_{k+1}$ which is not an edge of $\str{G}_k$ considering
its witness $w$ (i.e. vertex $w$ such that the edges $u,w$ and $v,w$ implied the addition of the edge $u,v$) and extending
$\str{O}_k$ by a new vertex $w'$ and edges $d(u,w')=d(u,w)$ and $d(v,w')=d(v,w)$.
One can verify that the completion algorithm will fail to complete $\str{O}_k$ the same way
as it failed to complete $\str{O}_{k+1}$ and moreover there is a homomorphism $\str{O}_{k+1}\to \str{G}_k$.

At the end of this procedure we obtain $\str{O}_0$, a subgraph of $\str{G}$, that
has no completion into $\mathcal A^\delta_{K,C}$.
The bound on the size of the cycle follows from the fact that only $\delta$ steps
of the algorithm are actually changing the graph and each time every edge may
introduce at most one additional vertex.

Let $\mathcal O$ consist of all edge-labelled cycles with at most $2^\delta 3$
vertices that are not completable in $\mathcal A^\delta_{K,C}$. Clearly $\mathcal O$ is finite. To check that $\mathcal O$ is a set of obstacles it remains
to verify that there is no $\str{O}\in \mathcal O$ with a homomorphism
to some $\str{M}\in \mathcal A^\delta_{K,C}$. Denote by $\mathcal{O}'$ the set of all homomorphic images
of structures in $\mathcal{O}$ that are not completable in $\mathcal A^\delta_{K,C}$.
Assume, to the contrary, the existence of such an
$\str{O}=(O,d)\in \mathcal{O}'$ and $\str{M}=(M,d')$ and a homomorphism $f\colon \str{O}\to\str{M}$ and among all those choose one minimising 
$\lvert \str{O}\rvert$. It follows that $\lvert \str{O}\rvert-\lvert \str{M}\rvert=1$.
Denote by $x,y$ the pair of vertices identified by $f$. Let $\str{O}'=(O,d'')$  be
a metric graph such that $d''(z,z')=d(f(z),f(z'))$ for every pair $\{z,z'\}\neq \{x,x'\}$.
It follows that, because $\mathcal A^\delta_{K,C}$ has the strong amalgamation property, also
$\str{O}'=(O,d'')$ has a completion in $\mathcal A^\delta_{K,C}$.
\end{proof}
\label{sec:applications}
We now have all the tools necessary to prove the main theorems of this paper.
\begin{proof}[Proof of Theorem~\ref{thm:EPPA}]
We showed in Lemma~\ref{lem:obstacles} that $\mathcal A^\delta_{K,C}$ is contained in $\Forb(\mathcal O)$, where $\mathcal O$ is a subset of the set of edge-labelled cycles with at most $2^\delta3$ edges; the \Fraisse{} limit of $\mathcal A^\delta_{K,C}$ now can play the \role{} of $M$ in Theorem~\ref{thm:herwiglascar}. Because the magic completion preserves automorphisms (by Theorem~\ref{thm:magiccompletion}) the conclusion follows.
\end{proof}
\begin{proof}[Proof of Theorem~\ref{thm:ramsey}]
We have shown that $\mathcal A^\delta_{K,C}$ is locally finite in Lemma \ref{lem:obstacles}; since it is an amalgamation class, it is hereditary. Also, the classes studied here are the ones with primitive \Fraisse{} limits, and amalgamation is strong. Therefore we can apply Theorem~\ref{thm:localfini} to conclude that $\mathcal A^\delta_{K,C}$ is a Ramsey class.
\end{proof}

The analysis presented in this paper can decide whether a class has coherent EPPA (a strengthening of EPPA by Siniora and Solecki~\cite{solecki2009,Siniora} where the extensions have to compose whenever the partial automorphisms do) or the Ramsey property when applied to slight modifications of the generalised algorithm in nearly all other classes of metrically homogeneous graphs with finite diameter in Cherlin's catalogue, as well as most of the infinite-diameter cases. A full account of this fact will appear in~\cite{Aranda2017}. We also have reason to believe that the algorithmic approach presented here can be applied to more general classes of relational structures with forbidden configurations.

\section{Acknowledgements}
A significant part of this research was done when the authors were participating in the Ramsey DocCourse programme, in Prague 2016--2017. The eighth author would like to thank Jacob Rus for his Python style guidance.
Similar results were also independently obtained by Rebecca Coulson and will appear
in~\cite{Coulson}.
\bibliographystyle{amsplain}
\bibliography{ramsey.bib}
\end{document}